\documentclass[12pt]{amsart}
\usepackage{amscd,amssymb,graphics,color,a4wide,hyperref,verbatim}
\usepackage{graphicx}
\usepackage{psfrag} 
\usepackage{kbordermatrix}
\usepackage{soul} 

\usepackage{mathrsfs}
\input xy
\xyoption{all}

\renewcommand{\kbldelim}{(} 
\renewcommand{\kbrdelim}{)}

\newcommand{\op}[1]{\operatorname{#1}}
\newcommand{\cone}{\operatorname{{Cone}}}

\DeclareFontFamily{U}{rsf}{}
\DeclareFontShape{U}{rsf}{m}{n}{
  <5> <6> rsfs5 <7> <8> <9> rsfs7 <10->  rsfs10}{}
\DeclareMathAlphabet{\mathscr}{U}{rsf}{m}{n}

\newtheorem{theorem}{Theorem} 
\newtheorem{lemma}[theorem]{Lemma}
\newtheorem{lemma-definition}[theorem]{Lemma-Definition}
\newtheorem{proposition}[theorem]{Proposition}

\newtheorem{assumption}[theorem]{Assumption}

\theoremstyle{definition}
\newtheorem{definition}[theorem]{Definition}

\newtheorem{example}[theorem]{Example}
\newtheorem{examples}[theorem]{Examples}

\theoremstyle{remark}
\newtheorem{remark}[theorem]{Remark}

\numberwithin{equation}{section}

\newcommand{\DD} {\mathbb{D}}

\newcommand{\NN} {\mathbb{N}}
\newcommand{\ZZ} {\mathbb{Z}}
\newcommand{\QQ} {\mathbb{Q}}
\newcommand{\RR} {\mathbb{R}}
\newcommand{\CC} {\mathbb{C}}

\newcommand{\HH} {\mathbb{H}}
\newcommand{\PP} {\mathbb{P}}
\renewcommand{\AA} {\mathbb{A}}

\newcommand {\shD} {\mathcal{D}}

\newcommand {\shF} {\mathcal{F}}

\newcommand {\shL} {\mathcal{L}}
\newcommand {\shM} {\mathcal{M}}

\newcommand {\shO} {\mathcal{O}}

\newcommand {\shT} {\mathcal{T}}

\newcommand {\shX} {\mathcal{X}}

\newcommand {\Blo} {\operatorname{Blo}}

\newcommand {\coker} {\operatorname{coker}}

\newcommand {\diag} {\operatorname{diag}}

\newcommand {\dlog} {\operatorname{dlog}}

\newcommand {\dual} {{\vee}}

\newcommand {\eps} {\varepsilon}

\newcommand {\Gr} {\operatorname{Gr}}

\newcommand {\Hom} {\operatorname{Hom}}

\newcommand {\hra} {\hookrightarrow}

\newcommand {\id} {\operatorname{id}}
\newcommand {\im} {\operatorname{im}}

\renewcommand {\ker } {\operatorname{ker}}

\newcommand {\len} {\operatorname{len}}

\newcommand {\llog} {{\operatorname{log}}}

\newcommand {\lra} {\longrightarrow}

\newcommand {\pr} {\operatorname{pr}}

\newcommand {\ra} {\to}

\newcommand {\rank} {\operatorname{rank}}

\newcommand {\Sing} {\operatorname{Sing}}

\newcommand {\Spec} {\operatorname{Spec}}

\newcommand {\vol} {\operatorname{vol}}

\newcommand {\T} {\shT}

\def\mydate{\ifcase\month \or January\or February\or March\or
April\or May\or June\or July\or August\or September\or October\or 
November\or December\fi \space\number\day,\space\number\year}

\newlength{\picwidth} 
\newlength{\miniwidth} 
\newlength{\nanowidth} 
\newlength{\melowidth} 
\newlength{\leftminiwidth} 
\newlength{\rightminiwidth} 
\newlength{\minipagewidth} 

\begin{document}
\def\mapright#1{\smash{
 \mathop{\longrightarrow}\limits^{#1}}}
\def\mapleft#1{\smash{
 \mathop{\longleftarrow}\limits^{#1}}}
\def\exact#1#2#3{0\to#1\to#2\to#3\to0}
\def\mapup#1{\Big\uparrow
  \rlap{$\vcenter{\hbox{$\scriptstyle#1$}}$}}
\def\mapdown#1{\Big\downarrow
  \rlap{$\vcenter{\hbox{$\scriptstyle#1$}}$}}
\def\dual#1{{#1}^{\scriptscriptstyle \vee}}
\def\invlim{\mathop{\rm lim}\limits_{\longleftarrow}}
\def\rto{\raise.5ex\hbox{$\scriptscriptstyle ---\!\!\!>$}}

\input epsf.tex
\title
[Perverse Curves and Mirror Symmetry]
{Perverse Curves and Mirror Symmetry}

\begin{abstract}
This work establishes a subtle connection between mirror symmetry for Calabi-Yau threefolds and that of curves of higher genus.
The linking structure is what we call a perverse curve. 
We show how to obtain such from Calabi-Yau threefolds in the Batyrev mirror construction and prove that their Hodge diamonds are related by the mirror duality.
\end{abstract}

\author{Helge Ruddat}

\address{JGU Mainz, Institut f\"ur Mathematik, Staudingerweg 9, 55099 Mainz, Germany}
\email{ruddat@uni-mainz.de}

\maketitle
\setcounter{tocdepth}{1}
\bigskip


\section*{Introduction}
Evidence for mirror symmetry to apply to varieties of positive Kodaira dimension has been given in \cite{Sei08},\cite{KKOY09},\cite{Ef09},\cite{GKR12},\cite{AAK12}.
Gross, Katzarkov and the author suggest in \cite{GKR12} that the mirror dual of a curve $Z$ of genus $g\ge 2$ is a union of $3g-3$ projective lines that meet in $2g-2$ points such that exactly three components meet in each point. 
By means of a duality of Landau-Ginzburg models in \cite{GKR12}, this reducible curve $\check Z$ comes together with a perverse sheaf $\shF_{\check Z}$ of vanishing cycles supporting a cohomological mixed Hodge complex of sheaves. 
The perverse sheaf should be thought of as an analogue of the constant sheaf $\shF_Z=\ZZ_Z$ on $Z$ which supports the cohomological Hodge complex that computes the usual Hodge structure of $(Z,\shF_Z)$. 
It was shown that $(Z,\shF_Z)$ and $(\check Z,\shF_{\check Z})$ have dual Hodge diamonds in the context of a construction where $Z$ embeds as an ample divisor in a toric surface.
For the notion of a cohomological mixed Hodge complex, we refer to \cite[III, 8.1.6)]{DelTH}.
\begin{definition} 
A \emph{perverse curve} is a pair $(Z,\shF_Z)$ of a (possibly reducible) curve $Z$ with perverse sheaf $\shF_Z$ supporting a cohomological mixed Hodge complex. 
\end{definition}
There are higher dimensional analogues but this article focuses on curves. 
The baby example and basic building block of a perverse curve is the mirror dual of a pair of pants which is given by the singular locus of the union of coordinate hyperplanes in 
$\CC^3$ together with the sheaf of vanishing cycles for the function $xyz$ (the product of the coordinate functions). 
This local duality was argued via homological mirror symmetry in \cite{AAEKO11}, see also \cite{Sh10}. 
We expect that a Strominger-Yau-Zaslow version of mirror symmetry can be extended to higher genus curves fibering over a tropical base.
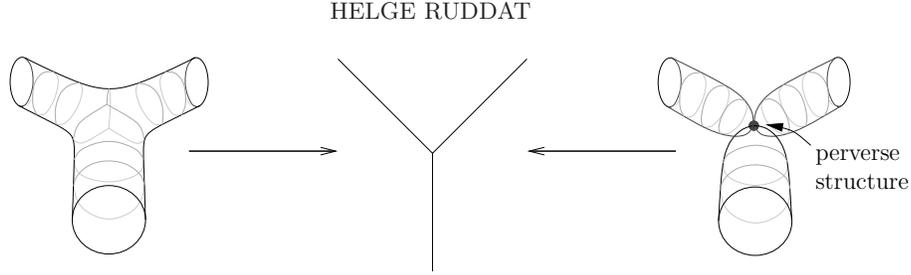
\begin{center}
\resizebox{0.8\textwidth}{!}{
\begin{picture}(0,0)%
\includegraphics{pairofpants.pstex}%
\end{picture}%
\setlength{\unitlength}{4144sp}%
\begingroup\makeatletter\ifx\SetFigFont\undefined%
\gdef\SetFigFont#1#2#3#4#5{%
  \reset@font\fontsize{#1}{#2pt}%
  \fontfamily{#3}\fontseries{#4}\fontshape{#5}%
  \selectfont}%
\fi\endgroup%
\begin{picture}(7532,1705)(1233,-4573)
\put(8027,-3684){\makebox(0,0)[lb]{\smash{{\SetFigFont{12}{14.4}{\familydefault}{\mddefault}{\updefault}{\color[rgb]{0,0,0}perverse}%
}}}}
\put(8027,-3909){\makebox(0,0)[lb]{\smash{{\SetFigFont{12}{14.4}{\familydefault}{\mddefault}{\updefault}{\color[rgb]{0,0,0}structure}%
}}}}
\end{picture}%

}
\end{center}

The main objective of this work is to relate mirror symmetry for varieties of general type, here in the form of perverse curves, to mirror symmetry of Calabi-Yau manifolds.
The critical locus of a Strominger-Yau-Zaslow (SYZ) fibration of a Calabi-Yau threefold shows a striking similarity to a perverse curve: not only is it a union of irreducible curves meeting in triple points, at least locally, the sheaf of vanishing cycles of the SYZ map supported on the critical locus gives a perverse sheaf of the same type as that appearing for a mirror dual of a curve of higher genus, cf. \cite{Gr01}, \cite{WR03}, \cite{GS03}, \cite{AAK12}.
There are two problems though. Firstly, a global definition of this perverse sheaf seems elusive because of monodromy in the SYZ fibration. Secondly, since the SYZ map is not holomorphic, we lack the structure of a cohomological mixed Hodge complex.

We show how to solve both of these problems in the presence of a two-dimensional linear system of reduced Calabi-Yau hypersurfaces in a normal ambient fourfold $\PP$ with the property that there are generators $X_0,X_1,X_2$ with $X_0$ simple normal crossing, $X_1,X_2$ smooth and such that the intersection of any subset of $\{X_0,X_1,X_2\}$ is simple normal crossing of the expected dimension. 
In particular, the Batyrev construction for a Calabi-Yau threefold gives rise to such a setup \cite{Ba94} by defining $X_0$ to be the complement of the open torus in a maximal projective crepant partial (MPCP) resolution $\PP$ of a toric Fano fourfold and $X_1,X_2$ as general hypersurfaces linearly equivalent to $X_0$.
We denote by $\shX\subset \PP\times\PP^1$ the pencil generated by $X_0,X_1$ and by $\shX'$ the one generated by $X_0,X_2$. 
Consider the intersection $\shD=(X_1\times\PP^1)\cap_{\PP\times\PP^1}\shX'$.
The induced map $\shD\ra\PP^1$ gives a degenerating family of surfaces near the origin. The central fibre $D_0=X_0\cap_{\PP} X_1$ coincides with the base locus of $\shX$.
Near $X_0$, the singular locus of $\shX$ contains the curve $Z=\Sing D_0=(\Sing X_0)\cap_\PP X_1$.
The degeneration $D_0=\lim_{s\to 0} D_s$ furnishes $Z$ with a perverse sheaf $\shF_Z$ supporting a cohomological mixed Hodge complex as we explain in \S\ref{S-topology},\S\ref{S-cohomology}.

\begin{theorem} \label{maintheorem}
Let $(Z,\shF_Z)$, $(\check Z,\shF_{\check Z})$ be perverse curves obtained by the above procedure respectively from mirror partners of the Batyrev construction of Calabi-Yau threefolds after a MPCP resolution of the ambient toric Fano varieties.
\begin{enumerate}
\item The Euler number of $(Z,\shF_Z)$ coincides with the Euler number of a general member of $\shX$. A similar statement holds for the duals.
\item Let $\Gamma$ denote the $1$-skeleton of the dual intersection complex of $D_0$ and $b_1(\Gamma)$ its first Betti number. Let $v,e$ be the number of vertices and edges of $\Gamma$. 
Note that $e$ coincides with the number of components of $Z$.
Let $n$ be the number of triple points in $Z$ and $g$ the sum of the genera of the components of $Z$. We have
$$
\begin{array}{rcccccccl}
h^{1,0}(Z,\shF_Z) &=& h^{0,1}(Z,\shF_Z) &=& n+g-b_1(\Gamma)&=&v-1+n+g-e, \\
h^{0,0}(Z,\shF_Z) &=& h^{1,1}(Z,\shF_Z) &=& e-b_1(\Gamma)&=&v-1.$$ 
\end{array}
$$
\item $h^{p,q}(Z,\shF_Z)= h^{1-p,q}(\check Z,\shF_{\check Z}).$
\end{enumerate}
\end{theorem}
\begin{proof} Item (1), (2), (3) will be proved in \S2.2 in theorem \ref{mainthm-1}, \ref{mainthm-2}, \ref{mainthm-3} respectively.
\end{proof}
Phrasing a homological mirror symmetry conjecture for $(Z,\shF_Z), (\check Z,\shF_{\check Z})$ currently fails by the absence of a definition of the Fukaya category of a perverse curve. Some progress towards the latter has been made by Auroux and Ganatra as well as Abouzaid and Auroux \cite{AA} using ambient Landau-Ginzburg models.

\begin{examples} 
\label{examples} The calculation for the following (1) and (2) can be found in \S\ref{appendix}.
\begin{enumerate}
\item The perverse curves in the quintic threefold and its mirror dual have the following Hodge diamonds.
\begin{center}
\[
\begin{array}{p{5cm}p{5cm}}
\xy
(15,0)*{}="A"; (0,12)*{}="B"; (15,24)*{}="C"; (30,12)*{}="D";
"A"; "B" **\dir{-};
"A"; "D" **\dir{-};
"C"; "B" **\dir{-};
"C"; "D" **\dir{-};
(15,12)*{
\begin{array}{ccc}
&4\\
104&&104\\
&4
\end{array} 
};
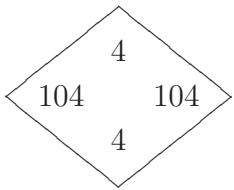
(15,-6)*{\hbox{\small quintic perverse curve}};
\endxy
&
\xy
(15,0)*{}="A"; (0,12)*{}="B"; (15,24)*{}="C"; (30,12)*{}="D";
"A"; "B" **\dir{-};
"A"; "D" **\dir{-};
"C"; "B" **\dir{-};
"C"; "D" **\dir{-};
(15,12)*{
\begin{array}{ccc}
&104\\
4&&4\\
&104
\end{array} 
};
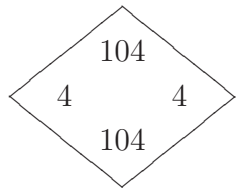
(15,-6)*{\hbox{\small quintic dual perverse curve}};
\endxy
\end{array}
\] 
\end{center}
It was noticed by W. Ruan that the curve $Z=(\Sing X_0)\cap X_1$ (without perverse sheaf structure) for the quintic $X_1$ determines the pencil generated by $X_1$ and $X_0$ uniquely \cite[Thm 3.1]{WR99}. 
Later Gross and Siebert generalized such a result by showing that the log structure on $X_0$ is determined by $Z$ which in turn reproduces the 1-parameter family $\shX$ under a rigidity assumption on $Z$ \cite{GS03},\cite{GS11}, cf. \cite{Ru10}.

\item The perverse curves in the $(2,2,2,2)$-hypersurface in $(\PP^1)^4$ and its Batyrev mirror dual have the following Hodge diamonds respectively.
\begin{center}
\[
\begin{array}{p{5cm}p{5cm}}
\xy
(15,0)*{}="A"; (0,12)*{}="B"; (15,24)*{}="C"; (30,12)*{}="D";
"A"; "B" **\dir{-};
"A"; "D" **\dir{-};
"C"; "B" **\dir{-};
"C"; "D" **\dir{-};
(15,12)*{
\begin{array}{ccc}
&7\\
71&&71\\
&7
\end{array} 
};
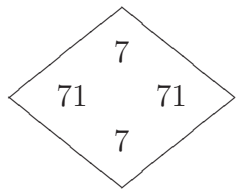
(15,-6)*{\hbox{\small $(2,2,2,2)$ perverse curve}};
\endxy
&
\xy
(15,0)*{}="A"; (0,12)*{}="B"; (15,24)*{}="C"; (30,12)*{}="D";
"A"; "B" **\dir{-};
"A"; "D" **\dir{-};
"C"; "B" **\dir{-};
"C"; "D" **\dir{-};
(15,12)*{
\begin{array}{ccc}
&71\\
7&&7\\
&71
\end{array} 
};
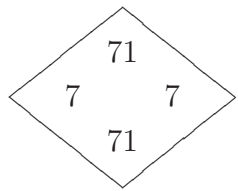
(15,-6)*{\hbox{\small $(2,2,2,2)$ dual perverse curve}};
\endxy
\end{array}
\] 
\end{center}
\item Schoen's Calabi-Yau threefold is obtained as a fibred product of two rational elliptic surfaces and was studied in \cite{Gr05} from a toric degeneration point of view. 
Its mirror dual is of the same type; the Hodge numbers are $h^{1,1}=h^{2,1}=19$. 
The perverse curve on either side consists of $24$ disjoint smooth elliptic curves, so $h^{i,j}=24$ for $0\le i,j\le 1$. 
We see that the statement of Thm.~\ref{maintheorem} holds for this example even though it doesn't fit in the Batyrev- but Batyrev-Borisov-duality \cite{BB94}.
\end{enumerate}
\end{examples}  
In the absence of an ambient space, we expect that perverse curve structures can be constructed, possibly under some conditions, by patching Landau-Ginzburg models similar to the techniques introduced in \cite{Jo13}.

\subsection*{Acknowledgments}  
We would like to thank Denis Auroux's seminar audience as well as Hertling-Sevenheck's workshop participants for their interest and critical remarks. 
Parts of this work were funded by the Carl-Zeiss-Foundation, DFG grants SFB-TR-45, RU 1629/2-1 and a Fields Postdoctoral Fellowship.

\section{Perverse curves from a normal crossing degeneration of surfaces}
\label{S-topology}
Let $\DD$ denote the unit disc. We say a proper holomorphic map $f:\shD \ra\DD$ is a normal crossing degeneration if $\shD$ is smooth, $D_0=f^{-1}(0)$ is a normal crossing divisor in $\shD$ and $f$ is smooth outside of $D_0$. 
Let $t_0\neq 0$ be a nearby value, $D_{t_0}=f^{-1}(t_0)$ and $r:D_{t_0}\ra D_0$ a retraction map, then we define the perverse sheaf of vanishing cycles
$$\shF_Z=\phi_f\ZZ[1]:=\cone(\ZZ_{D_0}\ra Rr_*\ZZ_{D_{t_0}})[1]$$
which is supported on $Z:=\Sing D_0$.
One can replace the non-canonical map $r$ by the canonical map $\tilde \shD^*\ra \shD$ \'a la Deligne where $\tilde \shD^*$ is the universal cover of $\shD^*=\shD\setminus D_0$. One then pulls back the resulting sheaf from $\shD$ to $D_0$. 
Another canonical choice the we are going to give in detail in \S\ref{S-global-topology} is the map
$r: (D_0)_{\log,1}\ra D_0$ from the fibre over $1$ of the Kato-Nakayama space associated to a log smooth morphism obtained from the map of pairs $f:(\shD,D_0)\ra (\DD,0)$.
All these give (quasi-)isomorphic sheaves $\shF_Z$.

Deligne \cite{DelTH} and Steenbrink \cite{St75} constructed a cohomological mixed Hodge complex of sheaves supported on $\ZZ_{D_0}$ and $Rr_*\ZZ_{D_{t_0}}$ respectively. 
By taking the mixed cone of these, one obtains a cohomological mixed Hodge complex of sheaves supported on $\shF_Z$ which furnishes the sheaf of vanishing cycles of $f$ with a mixed Hodge structure, see \cite{GKR12} and references therein for further details. Most notably $\shF_Z\otimes\CC=\phi_f\CC[1]$, $F^k(\shF_Z\otimes\CC)=F^{k+1}\phi_f\CC[1]$ and $W^k(\shF_Z\otimes\CC)=W^{k+1}\phi_f\CC[1]$.

\begin{lemma-definition} 
\label{Poincare-pc}
Let $\dim D_0=2$ then $(Z,\shF_Z)$ is a perverse curve. 
The Hodge numbers defined by $h^{p,q}(Z,\shF_Z):=\dim\Gr_F^p\HH^{p+q}(Z,\shF_Z\otimes\CC)$ (i.e. by ignoring the weight filtration) satisfy Poincar\'e duality, i.e. 
$h^{0,0}=h^{1,1}$, $h^{1,0}=h^{0,1}$.
\end{lemma-definition}
\begin{proof} This follows directly from Prop.~\eqref{Hodge-pc}. A second proof is the following: the given shifts from $\phi_f$ to $\shF_Z$ imply the first and last equality in the chain
\begin{equation}
\resizebox{\textwidth}{!}{
$h^{p,q}\HH^i(Z,\shF_Z)=h^{p+1,q+1}\HH^{i+1}(D_0,\phi_f\CC)=h^{3-(p+1),3-(q+1)}\HH^{4-(i+1)}(D_0,\phi_f\CC)=h^{1-p,1-q}\HH^{2-i}(Z,\shF_Z)$ \nonumber
}
\end{equation}
whereas the middle one is \cite[Lem. 4.7,(5)]{GKR12} via $Y=D_0$, $Z=\Sing Y$, $n=3$.
\end{proof}

When $\dim D_0=2$, the Hodge numbers of $(Z,\shF_Z)$ are determined by the topology just like for usual algebraic curves.
By definition, for a contractible open subset $U\subseteq Z$, we have
\begin{equation}
H^i(U,\shF_Z)=\left\{\begin{array}{ll}H^{i+1}(r^{-1}(U),\ZZ)&\hbox{ for }i\ge 0\\0&\hbox{ for }i<0.\end{array}\right.
\label{F-is-r-inverse}
\end{equation}

\subsection{Global topology}
\label{S-global-topology}
Consider the log structure $\alpha:\shM_{(\shD,D_0)}=j_*\shO^\times_{\shD\setminus D_0}\cap \shO_\shD \hra \shO_\shD$ on $\shD$ where $j:\shD\setminus D_0\hra \shD$ denotes the usual embedding and $\alpha$ is the natural inclusion. It is of Deligne-Faltings-type \cite[Complement 1]{Ka89}, i.e. for some $N\in\NN$ there is a map of monoid sheaves $\phi:\NN^N\ra \shM_{(\shD,D_0)}/\shO^\times_\shD$ that lifts \'etale locally to a chart of $\shM_{(\shD,D_0)}$. Equivalently, it can be given by a set of $N$ line bundles $\shL_1,...,\shL_N$ on $\shD$ with homomorphisms $s_i:\shL_i\ra\shO_\shD$. 
Indeed, let $\shL_i$ be the line bundle associated to the $\shO^\times_\shD$-torsor $\shL_i^\times=\pi^{-1}(\phi(e_i))$ where $\pi:\shM_{(\shD,D_0)}\ra\shM_{(\shD,D_0)}/\shO^\times_\shD$ denotes the natural projection and $e_i$ is the $i$th generator of $\NN^N$. Then take $s_i$ to be the map induced by $\alpha$. 
Conversely, given $s_i:\shL_i\ra\shO_\shD$, denote by $s$ the map of monoid sheaves
\begin{equation} 
s:T_{\shO_\shD^\times} \left(\bigoplus_{i=1}^N \shL^\times_i\right) \ra \shO_\shD
\label{tensor-construction}
\end{equation}
given at degree one by
$\prod s_i:\bigoplus_{i=1}^N \shL^\times_i\ra\shO_\shD$ and $T_{\shO_\shD^\times}$ means taking the tensor algebra of an $\shO_\shD^\times$-module.
We reconstruct $\shM_{(\shD,D_0)}$ as the log structure associated to the pre-log structure $s$, i.e. 
$\shM_{(\shD,D_0)}=\left.T_{\shO_\shD^\times} \left(\bigoplus_{i=1}^N \shL^\times_i\right)\right/\sim$
where we define $m\sim n$ iff $am=bn$ for some $a,b\in s^{-1}(c)$ for some $c\in\shO^\times_\shD$.

Let $D_1,...,D_N$ be an enumeration of the components of $D_0$.
In our case, the map $\phi$ is given by the orders of vanishing of elements of $\shM_{(\shD,D_0)}$ along $D_1,...,D_N$, $\shL_i=\shO_\shD(-D_i)$ and $s_i$ is the natural embedding $\shO_\shD(-D_i)\ra\shO_\shD$.
The map $f:\shD\ra\DD$ gives a section $f\in\shM_{(\shD,D_0)}$ upon picking a coordinate on the disk and $\pi(f)$ is the image under $\phi$ of the diagonal element $\sum_i e_i\in \NN^N$. 
The choice of coordinate becomes irrelevant when one pulls back the log structure from $\shD$ to $D_0$ as we do shortly in order to construct the canonical retraction $r$.
\begin{remark}
\label{normal-bundles}
For $i\neq j$, let $D_{i,j}=D_i\cap D_j$ and $D^\circ_{i,j}=D_{i,j} \setminus \bigcup_{k\neq i,j} D_k$.
The existence of a section $f\in\shL^\times_1\otimes...\otimes\shL^\times_N$ places restrictions on the line bundles $\shL_i$ as it trivializes their tensor product. 
At the locus $D^\circ_{i,j}$ all $\shL_k$ for $k\neq i,j$ can be trivialized and we get $(\shL_i^\times)|_{D^\circ_{i,j}}\otimes(\shL_j^\times)|_{D^\circ_{i,j}}\cong\shO^\times_{D_{i,j}}$, so the normal bundles along $D^\circ_{i,j}$ in $D_i$ and $D_j$ are dual to another.
We may replace $D^\circ_{i,j}$ in the statement by the real oriented blow-up $\Blo_P D_{i,j}$ of $P=D_{i,j}\cap  \bigcup_{k\neq i,j} D_k$ inside $D_{i,j}$.
\end{remark}

\begin{definition} 
\label{def-KNspace}
The \emph{Kato-Nakayama space} $\shD_\llog$ is the set of pairs $(x,\sigma)$ with $x\in \shD$ and $\sigma\in\Hom(\shM_{(\shD,D_0),x},S^1)$ such that $\sigma\circ\alpha$ sends $h\in\shO_\shD^\times$ to 
$\frac{h}{||h||}$; for further details, see \cite[after Def.3.4]{NO10}; cf. \cite{KN99} and \cite{RSTZ12}.
\end{definition}

\begin{example}
\label{KN-orblow}
If $\shM_{(D_{i,j},P)}$ denotes the divisorial log structure on $D_{i,j}$ with respect to $P$ then $\Blo_P D_{i,j}$ is its Kato-Nakayama space.
\end{example}

Let $\rho :\shD_\llog\ra\shD$ denote the natural projection $(x,\sigma)\mapsto x$ which is an isomorphism away from $D_0$. By \cite[Theorem 5.1]{NO10}, the map $\shD_\llog\ra\DD_\llog$ induced by the map of log spaces $f:(\shD,\shM_{(\shD,D_0)})\ra(\DD,\shM_{(\DD,0)})$ is a topological fibre bundle. We have a commutative diagram.
\begin{equation}
\xymatrix@C=30pt
{
\shD_\llog\ar_{f_\llog}[d]\ar^\rho[r]& \shD\ar^f[d] \\
\DD_\llog\ar[r] & \DD
}
\label{Dlog}
\end{equation}
We pull back the log structures $\shM_{(\shD,D_0)}$ to $D_0$ and $\shM_{(\DD,0)}$ to $0$. 
Constructing the resulting map on Kato-Nakayama spaces yields
\begin{equation}
\xymatrix@C=30pt
{
D_{0,\llog}\ar_{(f|_{D_0})_\llog}[d]\ar^\rho[r]& D_0\ar^{f|_{D_0}}[d] \\
\{0\}_\llog\ar[r] & \{0\}
}
\label{D0log}
\end{equation}
which is the pullback of \eqref{Dlog} to $\{0\}_\llog\ra \{0\}$. This map is the projection 
$$\{0\}\times \Hom(\NN,S^1)\ra \{0\}.$$ 
Let $1\in \Hom(\NN,S^1)$ be the trivial map which we identify also with $(0,1)\in (\{0\}\times \Hom(\NN,S^1))$. 
Pulling back the left column in \eqref{D0log} to $1$ yields
$$
\xymatrix@C=30pt
{
(D_{0,\llog})_1\ar_{(f|_{D_0})_\llog}[d]\ar[r]^r& D_0\ar^{f|_{D_0}}[d] \\
\{1\} \ar[r] & \{0\}
}
$$
where we find the canonical retraction map $r$ from the nearby to special fibre as the top vertical map. We have
$$(D_{0,\llog})_1= \{(x,\sigma)\in D_{0,\llog} \mid \sigma(f)=1\}.$$
We are interested in the restriction of $r$ to $Z=\Sing D_0$ which again can realized by pulling back the log structure to $Z$.
This pullback is most easily understood by pulling back the line bundles $\shL_i$ to $Z$ and performing the constructing in $\eqref{tensor-construction}$.
We extend Rem.~\ref{normal-bundles} to the following lemma.

\begin{lemma}
\label{lemma-orientation-S1-bundles}
The map $r^{-1}(D^\circ_{i,j})\ra D^\circ_{i,j}$ is the circle bundle associated to $\shL_i|_{D^\circ_{i,j}}=\shO_{D_i}(-D_{i,j})|_{D^\circ_{i,j}}$ or $\shL_j|_{D^\circ_{i,j}}=\shO_{D_j}(-D_{i,j})|_{D^\circ_{i,j}}$ depending on a choice of orientation of the circle bundle. Such a choice can be deduced from the ordering $i<j$ versus $j<i$. 
In particular $R^1(r|_{r^{-1}(D^\circ_{i,j})})_*\ZZ\cong\ZZ_{D^\circ_{i,j}}$ with the choice of such an isomorphism depending on the orientation.
\end{lemma}

\subsection{Local topology}
\label{S-local-topology}
Locally on its source, $f:\shD\ra\DD$ can be given as $f=z_1\cdot...\cdot z_s$ where $z_i$ are local equations of components of $D_0$. 
Let $\diag_s:\NN\ra\NN^s\times\NN^{n-s}$ denote product of the diagonal embedding into $\NN^s$ with the trivial map to $\NN^{n-s}$. 
We may describe $f$ locally as
$$\Spec\CC[\NN^s\times\NN^{n-s}]\ra \Spec\CC[\NN]$$
induced by $\diag_s$ and the log structure is induced by the chart $\NN^s\ra \CC[\NN^s\times\NN^{n-s}]$.
In this local description, the diagram \eqref{Dlog} becomes on closed points
$$
\xymatrix@C=30pt
{
\Hom(\NN^s,\RR_{\ge 0}\times S^1)\times \Hom(\NN^{n-s},\CC) \ar_{f_\llog}[d]\ar[r]& \Hom(\NN^s\times\NN^{n-s},\CC) \ar^f[d] \\
\Hom(\NN,\RR_{\ge 0}\times S^1)\ar[r] & \Hom(\NN,\CC)
}
$$
where each term is a $\Hom$ of commutative monoids, the vertical maps are induced by $\diag_s$ and the horizontal maps are induced by the monoid surjection
$\RR_{\ge 0}\times S^1\ra\CC$ realizing the real oriented blow-up of $\CC$ in the origin.
Following through the constructions of the previous section, we may describe $r$ locally as the map
$$\begin{array}{c}\{\phi\in \Hom(\NN^s,\RR_{\ge 0}\times S^1)\,|\, \phi(\diag_s(1))=(0,1)\}\times \Hom(\NN^{n-s},\CC) \qquad\qquad\qquad\\
\qquad\qquad\qquad \stackrel{r}{\lra} \{\phi\in\Hom(\NN^s,\CC)\,|\, \phi(\diag_s(1))=0\}\times \Hom(\NN^{n-s},\CC)
\end{array}$$
Denoting the trivial map by $0$ (the origin on the right hand side), we find 
$$r^{-1}(0)=\{\phi\in\Hom(\NN^s,S^1)\,|\, \phi(1,1,...,1)=1\} \cong (S^1)^{s-1}.$$
Representing $c\in S^1$ by $c=e^{2\pi i\theta}$ for $\theta\in[0,1)$, we have 
\begin{equation}
r^{-1}(0)=\{(\theta_1,...,\theta_s)\in[0,1)^s\,|\,\textstyle\sum_{i=1}^s\theta_i\in\ZZ\}.
\label{r-inv-is-torus}
\end{equation}
and more generally for $p=(r_1e^{2\pi i\alpha_1},...,r_ne^{2\pi i\alpha_n})$,
\begin{equation}
r^{-1}(p)=\{(\theta_1,...,\theta_s)\in[0,1)^s\,|\,\textstyle\sum_{i=1}^s\theta_i\in\ZZ,\ \theta_i=\alpha_i\hbox{ if }r_i>0\}.
\label{rinv}
\end{equation}
We used a choice of coordinates $z_1,...,z_s,...,z_n$ in the description given here for the sake of explicitness. The analogous constructions becomes coordinate-free if we use the monoid sheaf $\shM_{(\shD,D_0)}$ introduced in the previous section.

From now on, we restrict to the case $\dim D_0=2$ and study the topology from which we derive $\shF_Z$.
Let $s=3$, so we have a point $p$ where the maximal number of components of $D_0$ meet, moreover $p\in Z=\Sing D_0$. 
Let $Z_1,Z_2,Z_3$ denote the three components of $Z$ meeting in $p$. Working locally, we take them as discs, e.g. $Z_j\cong \{z_j=r_j e^{2\pi i\theta_j}\,|\,|r_j|<1\}$.
By \eqref{r-inv-is-torus}, $r^{-1}(p)\cong (S^1)^2$ is given as the anti-diagonal subtorus of $(S^1)^3$ as shown in Figure~\ref{antidiagtorus}.
There are three projections $q_i:r^{-1}(p)\ra S^1_{\theta_i}$ to the coordinate $S^1$'s of the cube turning the $2$-torus $r^{-1}(p)$ into a circle bundle over $S^1$ in three different ways.
\begin{figure}
\resizebox{0.5\textwidth}{!}{
\input{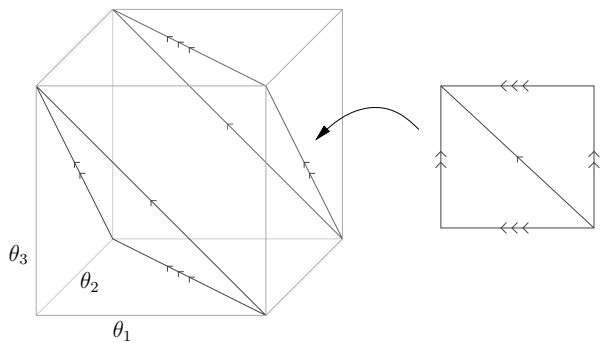}
}
\caption{The fibre of $r$ at a triple point of $Z$ with the cube depicting the fundamental domain of $(\RR/\ZZ)^3$}
\label{antidiagtorus}
\end{figure}
\begin{figure}
\resizebox{0.8\textwidth}{!}{
\input{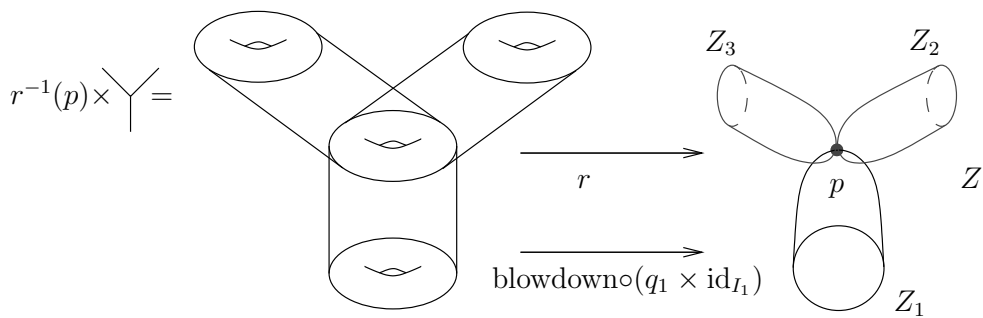}
}
\caption{Gluing $r^{-1}(Z)$ at a triple point of $Z$ where the components $Z_1,Z_2,Z_3$ meet.}
\label{gluetriple}
\end{figure}
Let $Y$ denote the union of three copies $I_1,I_2,I_3$ of the unit interval $[0,1)$ identified in $\{0\}$. 
We have a continuous map $Z\ra Y$ by sending $z\in Z_i$ to $|z|\in I_i$.
Then by \eqref{rinv}, $r:r^{-1}(Z_j\setminus\{0\})\ra Z_j\setminus\{0\}$ is the product of the circle bundle $q_j:r^{-1}(p)\ra S^1_{\theta_j}$ with $I_j\setminus \{0\}$ and these three pieces are glued over $p$ by inserting $r^{-1}(p)$. Applying $q_j\times \id_{I_j}$ to $r^{-1}(p)\times I_j$ yields $S^1_{\theta_j}\times I_j$, the real oriented blowup of $Z_j$ in $p$.
We thus have $r^{-1}(Z)=r^{-1}(p)\times Y$ and the projection $\pr_2$ to $Y$ factors through the map $r:r^{-1}(Z)\ra Z$ by identifying $Z$ with the equivalence relation on $r^{-1}(p)\times Y$ given by
$$x\sim x'\iff \begin{array}{c}\pr_2(x)=\pr_2(x')\hbox{ and if there is a unique $j$ s.t. }\\ \pr_2(x)\in I_j\hbox{ then } q_j(x)=q_j(x'),\end{array}$$
see Figure~\ref{gluetriple}.
We summarize the understanding of the geometry gained in this and the previous section in the following.
\begin{proposition} 
\label{summary-topology}
Let $Z_k$ be a component of $Z$. The map $r:r^{-1}(Z_k)\ra Z_k$ is the composition of 
\begin{enumerate} 
\item an orientable circle bundle $r^{-1}(Z_k)\ra \tilde Z_k$ over the real oriented blow-up $\tilde Z_k$ of $Z_k$ in the points where $Z_k$ meets other components of $Z$ with
\item the blow-down map $\tilde Z_k\ra Z_k$.
\end{enumerate}
The bundle $r^{-1}(Z_k)\ra \tilde Z_k$ is the pull-back of the circle bundle from $Z_k$ associated to the normal bundle of $Z_k$ in a component of $D_0$ containing $Z_k$ or its dual depending on the choice of orientation.
The space $r^{-1}(Z)$ is obtained by gluing the $r^{-1}(Z_k)$ along the real $2$-tori that lie over triple points in $Z$. 
One can construct $r^{-1}(Z)$ entirely from $\shM_Z$, the pullback of the log structure $\shM_{(\shD,D_0)}$ to $Z$ together with the global section given by $f:\shD\ra\DD$.
\end{proposition}
\begin{proof} On a component $Z_k=D_{i,j}$ of $Z$, we have three log structures contained in one another
$\shO^\times_{D_{i,j}}\subset \shM_{(D_{i,j},P)} \subset \shM_{(\shD,D_0)}|_{D_{i,j}}$
giving rise to maps of Kato-Nakayama spaces which are
$$ r^{-1}(Z) \ra \Blo_P Z_k \ra Z_k,$$
see Ex.~\ref{KN-orblow}. The remaining statements are Rem.~\ref{normal-bundles}, Lemma~\ref{lemma-orientation-S1-bundles} and what we said before the Proposition.
\end{proof}

\subsection{Cohomology}
\label{S-cohomology}
The sheaf $\shF_Z$ is given by $$\shF_Z=\cone (\ZZ\ra Rr_*\ZZ)[1]$$
which is quasi-isomorphic to the constant sheaf at a general point and at a triple point it has rank two in degree $0$ and rank one in degree $1$ as we deduce from \eqref{F-is-r-inverse} and \eqref{r-inv-is-torus}.
The cohomology of $\shF_Z$ can be computed from an open cover $\{U_i\,|\,1\le i\le l\}$ of $Z$ with $U_I=U_{i_1}\cap ...\cap U_{i_k}$ contractible for any subset 
$I=\{i_1,...,i_k\}\subseteq \{1,...,l\}$. We have a spectral sequence
$$E^{p,q}_1 = \bigoplus_{I,|I|={p+1}} H^{q+1}(r^{-1}(U_I),\ZZ)\Rightarrow \HH^{p+q}(Z,\shF_Z)$$
with $d_1$ given by the \v{C}ech differential. 
The $E_2$-term is
\begin{equation}
\begin{xy} 
*+{
\begin{array}{ccc}
H^{0}(Z,R^2r_*\ZZ)\\[3mm]
H^{0}(Z,R^1r_*\ZZ) & H^{1}(Z,R^1r_*\ZZ) & H^{2}(Z,R^1r_*\ZZ)
\end{array}
}*\frm{-,};
(-13,4)*{}="A"; (15,-1)*{}="B";
{\ar "A"; "B" };
(1,4)*{\scriptstyle d_2};
\end{xy}
\label{E2term}
\end{equation}
This coincides with the Leray spectral sequence of the map $r:r^{-1}(Z)\ra Z$ with bottom row removed and shifted by $-1$ in vertical direction.
We may assume that each triple point $p$ is contained in a unique open set $U_p$, so 
$$\rank H^i(U_p,\shF_Z)=
\left\{
\begin{array}{ll}
2 & i=0\\
1 & i=1.
\end{array}\right.$$
Observe that if we were to replace $\shF_Z$ by the constant sheaf $\ZZ_Z$ we would get a change of ranks only on $U_p$, namely $\rank H^0(U_p,\ZZ)=1$ and $\rank H^1(U_p,\ZZ)=0$. We deduce that the Euler numbers of the cohomology of $\ZZ_Z$ and $\shF_Z$ coincide.
Denoting the components of $Z$ by $Z_1,...,Z_M$, we thus find the following result for the Euler numbers of a perverse curve coming from a normal crossing degeneration of surfaces.
\begin{theorem}
\label{thm-euler-Z-ZF}
$e(Z,\shF_Z)=e(Z)=\sum_{i=1}^M e(Z_i)-2\#\{\hbox{triple points}\}$.
\end{theorem}

We next treat the Hodge structure. 
Let $D_1,...,D_N$ be an enumeration of the components of $D_0$ and $D^j=\coprod_{i_1<...<i_j} D_{i_1}\cap...\cap D_{i_j}$, so that
$$\begin{array}{c}
D^1=D_1\sqcup...\sqcup D_N,\\
D^2=Z_1\sqcup...\sqcup Z_M,\\
D^3=\{\hbox{triple points of }$Z$\}.\\
\end{array}$$
The alternating restriction map is defined as the map $\delta:H^k(D^2,\QQ)\ra H^k(D^3,\QQ)$ given by
\begin{equation}
\delta(\alpha)_{i_1,i_2,i_3}= (\alpha_{2,3}-\alpha_{1,3}+\alpha_{1,2})|_{D_{i_1}\cap D_{i_2}\cap D_{i_3}}
\label{defdelta}
\end{equation}
where $\alpha_{i,j}\in H^0(D_i\cap D_j,\CC)$.
Let $\delta^*$ denote the Poincar\'e dual map to $\delta$.
The diagram
\begin{equation}
\xymatrix@C=30pt
{
H^0(D^3,\QQ)\ar[r]^{\delta^*} & H^2(D^2,\QQ) \\
 & H^1(D^2,\QQ) \\
&H^0(D^2,\QQ)\ar[r]_\delta & H^0(D^3,\QQ)
}
\label{E2weightss}
\end{equation}
constitutes the $E_1$-term of the weight spectral sequence of rational level of the cohomological mixed Hodge complex of $\shF_Z$ where the indexing is such that $E^{0,0}_1=H^0(D^2,\QQ)$, see \cite[Lemma~4.7\,(3)-(4)]{GKR12} (to convert notation from there use $Y^i=D^i$ and note the shift $\Gr^W_i\shF_Z=\Gr^W_{i+1}\bar A$).
The columns from left to right in \eqref{E2weightss} give the graded pieces $\Gr^W_{1}, \Gr^W_{0}, \Gr^W_{-1}$ of the monodromy weight filtration on $\shF_Z$. 
\begin{remark} Matching \eqref{E2weightss} with the $\QQ$ scalar extension of \eqref{E2term},
note that the map $\delta^*$ in \eqref{E2weightss} is isomorphic to the only non-trivial differential $d_2\otimes\QQ$ in \eqref{E2term} 
whereas the total cohomology of the bottom three terms in \eqref{E2weightss} coincides with the remaining two terms in \eqref{E2term}.
\end{remark}

\begin{proposition}
\label{Hodge-pc}
We have the following decomposition in graded pieces by the Hodge filtration
\begin{equation}
\HH^i(Z,\shF_Z\otimes\CC)=\left\{\begin{array}{ll}
\coker(\delta^*)&\hbox{for }i=2\\
H^{1,0}\oplus H^{0,1} &\hbox{for }i=1\\
\ker\delta&\hbox{for }i=0
\end{array}\right.
\label{hodge-perv-curve}
\end{equation}
where $H^{1,0}$ and $H^{0,1}$ fit in exact sequences
$$
\begin{array}{rcccccccl}
0&\ra&\ker\delta^*                 &\ra& H^{0,1} &\ra& \oplus_{i=1}^M H^{0,1}(Z_i) &\ra& 0\\
0&\ra& \oplus_{i=1}^M H^{1,0}(Z_i) &\ra& H^{1,0} &\ra& \coker\delta                &\ra& 0
\end{array}
$$
induced by the weight filtration.
\end{proposition}
\begin{proof} 
Note that \eqref{E2weightss} degenerates at $E_2$ (see \cite[III, 8.1.9(iv)]{DelTH}). 
The only remaining issue then is the canonicity of the splitting on $\HH^1$ which we now focus on. It can be achieved using Deligne-splitting \cite[Lem-Def. 3.4]{PS08} 
by setting $H^{1,0}=I^{1,1}\oplus I^{1,0}$ and $H^{0,1}=I^{0,1}\oplus I^{0,0}$ which yields the given splitting because 
$I^{0,0}=W_0$, $I^{1,0}=F^1\cap\bar F^0\cap W_1$, $I^{0,1}=F^0\cap\bar F^1\cap W_1$ and $I^{1,1}=F^1\cap(\bar F^1+W_0)$.
\end{proof}
If all $Z_i$ are projective lines (e.g. see Example~\ref{ex-g2mir} below) then $Z$ is rigid and the only possible variation of the Hodge structure on $(Z,\shF_Z)$ arises from varying the extension class of the Hodge-Tate structures that give $\HH^1(S,\shF_Z\otimes\CC)$. We expect that further interesting variations of the above Hodge structure arise from some type of A-model Hodge structure mixing $\HH^0$ and $\HH^2$, see \cite[II,11]{DIP02}, provided one can find a suitable definition of quantum cohomology of a perverse curve.

The following example has been studied in \cite{GKR12}.
\begin{example}[Mirror dual of a genus two curve] 
\label{ex-g2mir}
\begin{figure}[ht]
\begin{minipage}[b]{0.45\linewidth}
\centering
\input{g2mirror.pstex_t}
\caption{Mirror dual of a genus two curve...}
\label{g2mirror}
\end{minipage}
\hspace{0.5cm}
\begin{minipage}[b]{0.45\linewidth}
\centering
\[
\xy
(13,0)*{}="A"; (0,12)*{}="B"; (13,24)*{}="C"; (26,12)*{}="D";
"A"; "B" **\dir{-};
"A"; "D" **\dir{-};
"C"; "B" **\dir{-};
"C"; "D" **\dir{-};
(13,12)*{
\begin{array}{ccc}
&2\\
1&&1\\
&2
\end{array} 
};
\endxy
\] 
\caption{...and its Hodge diamond}
\label{g2Hodge}
\end{minipage}
\end{figure}
Consider the regular function
$$w':\shX'=\Spec\CC[x,y,z,u,v]/(xy-z^2,uv-z^3)\ra\CC$$ given by $w'=x+y+z+u+v$.
Let $\shX$ be a crepant resolution of the blow up of the origin in $\shX'$. We denote the pullback of $w'$ to $\shX$ by $w$. By \cite{GKR12},~Ex.1.24, $w$ is an open subset of a type III degeneration of a K3 surface. $X_0=w^{-1}(0)$ is a normal crossing union of three rational surfaces $D_1,D_2,D_3$ and $Z:=D^2$ is a configuration of three $\PP^1$'s as in Figure~\ref{g2mirror}. 
The map $\delta$ is given by the matrix
$$
\kbordermatrix{&12&13&23&\\
123_1&1&-1&1\\
123_2&1&-1&1
}
$$
so it has rank one. By \eqref{hodge-perv-curve}, we have
$\HH^i(Z,\shF_Z)\cong\CC^2$ for i=1,2,3 and the Hodge diamond is that of a genus two curve rotated by a quarter turn, see Fig.~\ref{g2Hodge}.
\end{example}

\subsection{Cohomological mixed Hodge complex}
Let $(Z,\shF_Z)$ be a perverse curve coming from a normal crossing degeneration of surfaces as in \S\ref{S-topology}. The cohomological mixed Hodge complex on $\shF_Z$ can be given as follows. 
By \cite[Theorem 4.5.(1)]{GKR12}, we may replace the mixed cone by the cokernel of the injection $\CC\ra Rr_*\CC$ whose complex part is given by the double complex ($p,q\ge 0$)
$$\bar A^{p,q}=\Omega_{\shD}^{p+q+1}(\llog D_0) / W_{q+1}$$
where $\Omega_{\shD}^{p+q+1}(\llog D_0)$ denotes the sheaf of differential $(p+q+1)$-forms on $\shD$ with at worst logarithmic poles in $D_0$ and $W_{q+1}$ is the subsheaf given by
$\Omega_{\shD}^{q+1}(\llog D_0)\wedge \Omega_{\shD}^p$. 
Note that by definition, $\bar A^{p,q}=0$ for $p=0$. Moreover, since $\shD$ is a threefold, $\bar A^{p,q}=0$ for $p+q>2$. So only three terms of $\bar A^{\bullet,\bullet}$ are non-trivial, namely
\begin{equation}
\xymatrix@C=30pt
{
\bar A^{1,1}\\
\bar A^{1,0}\ar^d[r]\ar^{\wedge\dlog t}[u]& \bar A^{2,0}.
}
\label{barA1}
\end{equation}
The horizontal differential is the usual exterior derivative and the vertical one is wedging with $f^*(\frac{dt}t)$ for $t$ a coordinate on the base.
If $\bar A^\bullet$ denotes the total complex then 
$$\shF_Z\otimes\CC\cong \bar A^\bullet[1].$$
In order to better understand the terms of $\bar A^\bullet$, we consider the residue map
$$\op{res}_{r,I}:\Omega_{\shD}^{r}(\llog D_0) \ra \Omega^{r-q}_{D_I}(\llog (E_I))$$
for some $I=\{i_1,...,i_q\}\subset \{1,...,N\}$ with $D_I=D_{i_1}\cap...\cap D_{i_q}$ and 
$E_I=\bigcup_{j\not\in I} D_I\cap D_j$.
If $z_{i_j}$ is a local equation of $D_{i_j}$ then $\op{res}_{r,I}$ 
is defined to send $\frac{dz_{i_1}}{z_{i_1}}\wedge...\wedge \frac{dz_{i_q}}{z_{i_q}}\wedge \alpha+\beta$ 
with $\beta$ indivisible by $\frac{dz_{i_1}}{z_{i_1}}\wedge...\wedge \frac{dz_{i_q}}{z_{i_q}}$
to $\alpha|_{D_I}$ and is surjective, see \cite[Def. 4.5]{PS08}.
Summing $\op{res}_{p+q+1,I}$ over all size $q+2$ subsets $I$ of $\{1,...,N\}$ yields
$$\op{res}_{p,q}: \bar A^{p,q}\ra \bigoplus_{|I|=q+2}\Omega^{p-1}_{D_I}(\llog E_I)$$
which is well-defined since $W_{q+1}$ lies in the kernel of each $\op{res}_{p+q+1,I}$.
It restricts to an isomorphism
\begin{equation}
\op{res}_{p,q}|_{W_{q+2}}: W_{q+2}/W_{q+1}\ra \bigoplus_{|I|=q+2}\Omega^{p-1}_{D_I},
\label{res-iso}
\end{equation}
see \cite[Lem. 4.6]{PS08}. Note that $\bigoplus_{|I|=q}\Omega^{p}_{D_I}=\Omega^{p}_{D^{q}}$.
\begin{proposition}
\label{prop-residues-applied}
The term-wise residue map gives a quasi-isomorphism from \eqref{barA1} to
\begin{equation}
\xymatrix@C=30pt
{
\Omega^0_{D^3}\\
\Omega^0_{D^2}\ar^-d[r]\ar^\delta[u]& \tilde\Omega^1_{D^2}(\llog E^2)
}
\label{barA2}
\end{equation}
where $D^3=\coprod_{i_1<i_2<i_3} D_{i_1}\cap D_{i_2}\cap D_{i_3}$ is a union of points, $D^2=\coprod_{i_1<i_2} D_{i_1}\cap D_{i_2}$ a union of curves,
$E^2$ is the divisor on $D^2$ which is $E_I$ on $D_I$ as above and 
$\tilde\Omega^1_{D^2}(\llog E^2)$ is the subsheaf of $\Omega^1_{D^2}(\llog E^2)$ given by the property that for any $i_1<i_2<i_3$ and sections
$\alpha_{i_1,i_2}, \alpha_{i_1,i_3}, \alpha_{i_2,i_3}$ of $\Omega^1_{D_{{i_1},{i_2}}}(\llog E_{{i_1},{i_2}}), \Omega^1_{D_{{i_1},{i_3}}}(\llog E_{{i_1},{i_3}}), \Omega^1_{D_{{i_2},{i_3}}}(\llog E_{{i_2},{i_3}})$ respectively, we have that the residues of $\alpha_{i_1,i_2}, -\alpha_{i_1,i_3}, \alpha_{i_2,i_3}$ in $D_{i_1}\cap D_{i_2}\cap D_{i_3}$ coincide.
\end{proposition}
\begin{proof}
Note that the filtration $W_\bullet$ on $A^{1,1}$ as well as on $A^{1,0}$ has only one step, indeed $A^{1,1}=W_3/W_2$ and $A^{1,0}=W_2/W_1$ so the statement for these terms follows from the isomorphism \eqref{res-iso}. 
For the right term, let $z_1,z_2,z_3$ be defining equations for the components of $D_0$ at a triple point then a section of $\bar A^{2,0}$ near the triple point 
can uniquely be represented by $f d\log z_1\wedge d\log z_2 \wedge d\log z_3$ 
with $f=a_0+z_1g_1(z_1)+z_2g_2(z_2)+z_3g_3(z_3)$ for some $a_0\in\CC$, $g_i\in\CC\{t\}$.
Given $i\in\{1,2,3\}$, $I=\{1,2,3\}\setminus\{i\}$, the image under 
$\op{res}_{3,I}$ is $\eps(f d\log z_i)|_{D_I} = \eps(a_0\dlog z_i+g_i(z_i)dz_i)$ where $\eps=-1$ if $i=2$ and $\eps=1$ otherwise.

The vertical map in \eqref{barA2} is the alternating restriction map $\delta$ defined in \eqref{defdelta}, see \cite[\S11.2.5]{PS08}. The horizontal map in \eqref{barA2} is the usual exterior derivative.
\end{proof}

\begin{proposition} 
\label{depends-only-on-M_Z}
The cohomological mixed Hodge complex supported on $\shF_Z$ coming from the degeneration $f:\shD\ra\DD$ can be constructed entirely from the knowledge of $\shM_Z$, the pull back of the log structure $\shM_{(\shD,D_0)}$ to $Z$, together with the section $f$.
\end{proposition}
\begin{proof} 
We have seen that only pullback of logarithmic differential forms to $Z$ enter the definition of $\bar A^\bullet$. Such can be constructed from $\shM_Z$. 
On the other hand, by Prop. \ref{summary-topology}, the integral structure can be obtained from $\shM_Z$ as well. This is also true for the weight filtration, see \cite[\S4.4]{PS08}.
\end{proof}

\section{From linear systems of threefolds to degenerations of surfaces}
\label{S-linsys-to-pc}
\begin{figure}
\resizebox{0.42\textwidth}{!}{
\input{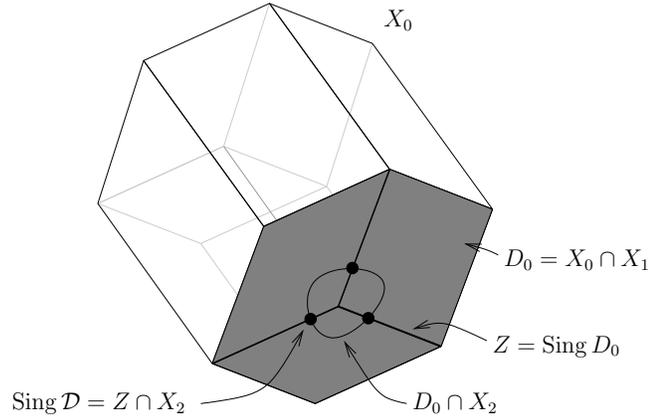}
}
\caption{Schematic view of $Z$ and $X_0\cap X_1\cap X_2$}
\label{x0x1x2}
\end{figure}
Let us recall the construction of perverse curves from the introduction and fill in some details.
\begin{definition}
A Cartier divisor $D$ in $\PP$ is \emph{simple normal crossing} if it is locally of the shape $\Spec[z_1,..,z_n]/(z_1\cdot...\cdot z_r)$ and its components don't self-intersect, i.e. the $z_i$ equate different components of the divisor. 
\end{definition}
\begin{assumption} \label{mainsetup}
We have anticanonical hypersurfaces $X_0, X_1, X_2\subset \PP$ with $X_0$ normal crossing, $X_1,X_2$ smooth.
Also $\bigcap_{i\in I} X_i$ is simple normal crossing of dimension $4-|I|$ for any $I\subset\{0,1,2\}$.
\end{assumption}
We assume that $X_0$ has points where at least three components meet and we require $\PP$ to be non-singular near $X_0$ except possibly at points of $X_0$ where $4$ components meet as these don't affect $\shD$. 
Allowing such singularities is necessary to include all of Batyrev's Calabi-Yau threefolds in our construction.
Let $f_i$ be a local equation for $X_i$, so $\shX=V(f_0+tf_1)$, $\shD=V(f_0+tf_2,f_1)$, $D_0=V(f_0,f_1)$.
By the assumptions on $X_0,X_1,X_2$, near $(\Sing X_0)\cap X_1\cap X_2$ we can find local coordinates $z_1,...,z_4$ such that 
$f_1=z_1$, $f_2=z_2$, $f_0=z_3z_4$ and so $\shD=\Spec\CC[z_1,z_2,z_3,z_4,t]/(z_1,z_2t-z_3z_4)$ and thus 
$$\Sing\shD=(\Sing X_0)\cap X_1\cap X_2$$
is a union of isolated ordinary double points. We choose a small resolution $\rho:\tilde\shD\ra\shD$. 
Note that this can be done by successively blowing up the components of $D_0$ and thus if $\shD$ is projective, we can assume $\tilde\shD$ to be projective as well.
Denoting $\tilde D_0=\rho^{-1}D_0$, $\rho$ gives an identification
$$\Sing\tilde D_0 = \Sing D_0 = Z$$
because at a singularity of $\shD$, $\Sing D_0$ is the locus where two components of $D_0$ meet and blowing up one of them doesn't change the locus of intersection of these two components.
The composition $\tilde\shD\stackrel{\rho}{\ra}\shD\ra\PP^1$ yields a normal crossing degeneration of surfaces. 
We thus obtain a perverse curve $(Z,\shF_Z)$ by \S\ref{S-topology}.
\begin{theorem} 
The perverse curve $(Z,\shF_Z)$ is independent of the choice of resolution $\rho:\tilde\shD\ra\shD$ and of the choice of $X_2$.
\end{theorem}
\begin{proof}
We have seen that $Z$ is independent of $\rho$. Since $Z=\Sing D_0$ and $D_0$ is independent of $X_2$, so is $Z$. It remains to show the independence of $\shF_Z$.
By Prop.~\ref{depends-only-on-M_Z}, the cohomological mixed Hodge complex on $\shF_Z$ depends only on $\shM_Z$ which in turn is determined by the $\shL_i$ that are given by normal bundles of the double intersection locus of $D_0$ in the components of $D_0$. From this we deduce that $\shF_Z$ is independent of the choice of $X_2$ as it so far only depends on $D_0=X_0\cap X_1$ and $X_2$ is only being used to ensure the existence of suitable gluings of the normal bundles to obtain the $\shL_i$ and the existence of the section $f$. 

Let us consider the local situation near a singularity $p$ of $\shD$. Let $D_1,D_2$ be the two components of $D_0$ meeting the singularity and let $\bar\shL_i$ be the normal bundle of $Z=D_1\cap D_2$ in $D_i$. If we blow up $D_1$ in $\shD$ and denote by $\tilde D_1$ its proper transforms, we have another normal bundle $\tilde\shL_1$ of $Z=\tilde D_1\cap D_2$ in $\tilde D_1$ and a calculation shows $\tilde\shL_1=\bar\shL_1(p)$. Similarly if we instead chose to blow up $D_2$ to resolve the singularity $p$, we would obtain a new normal bundle of $Z$ inside $\tilde D_2$ which is $\tilde\shL_2=\bar\shL_2(p)$. So we need to compare the two log structures given by the resulting pair of normal bundles $\shL_1,\shL_2$ that is either
$$\begin{array}{rclrcl}
\shL_1&=&\bar\shL_1,&\shL_2&=&\bar\shL_2(p)\quad \hbox{or}\\
\shL_1&=&\bar\shL_1(p)&\shL_2&=&\bar\shL_2.
\end{array}$$
The two resulting log structures $\shM_Z$ are different in general. However, the cohomological mixed Hodge complex will be independent. That this is true for its complex part is seen from \eqref{barA2} which only depends on $Z$ in fact. 
The relevant part of $\shM_Z$ for the integral part is $R^1(r|_{r^{-1}(Z)})_*\ZZ$. The choice of resolution $\tilde \shD$ could at most matter at $\Sing\shD$ which is a set of points contained in $Z$ away from the triple points.
On $C=D_{i,j}^\circ$ for some $i,j$ we have
$$\shF_Z=\cone(\ZZ_{C}\ra R(r|_{r^{-1}(C)})_*\ZZ)[1]$$
and this is canonically isomorphic to $\coker(\ZZ_{C}\ra (Rr|_{r^{-1}(C)})_*\ZZ)[1]=R^1(r|_{r^{-1}(C)})_*\ZZ$
which is by Lemma~\ref{lemma-orientation-S1-bundles} isomorphic to $\ZZ_C$ and the isomorphism is given by a choice of orientation of the circle bundle ${r^{-1}(C)}\ra C$. The choice of orientation is induced by the ordering $i<j$. 
Hence, $\shF_Z$ is independent of the choice of resolution $\tilde\shD$.
\end{proof}

\subsection{Batyrev's mirror construction}
A polytope $\Xi\subset \RR^n$ whose interior contains the origin and whose vertices are contained in $\ZZ^n$ is called \emph{reflexive} if the vertices of its polar dual polytope
$$\check\Xi=\{v\in\Hom(\RR^n,\RR) \mid v(x)\ge -1\hbox{ for all }x\in\Xi\}$$ are contained in $\Hom(\ZZ^n,\ZZ)$, see \cite[Def. 4.1.5]{Ba94}. 
In this case $\check\Xi$ is also reflexive with dual $\Xi$.
\begin{definition}
For $\Xi$ a reflexive polytope with polar dual $\check\Xi$, we call $\Xi,\check\Xi$ \emph{a dual pair of reflexive polytopes}.
\end{definition}
From now on let $\Xi$ be a reflexive polytope.
There is a natural inclusion-reversing duality of proper faces of $\Xi$ and of $\check \Xi$ sending a face $F\subset \Xi$ of dimension $d$ to the face 
$$\check F=\{v\in\check\Xi \mid v(x)=-1\hbox{ for all }x\in F\}$$ of dimension $n-1-d$, see \cite[Prop. 4.1.7]{Ba94}. 
A regular triangulation $\shT$ of the boundary of a reflexive polytope $\Xi$ is called MPCP if every simplex is elementary, i.e. its vertices are the only lattice points contained in it; similarly for a triangulation $\check \shT$ of $\check\Xi$. 
Let $\PP_\Xi$ denote the toric variety associated to $\Xi$. Its fan $\Sigma$ is given by the set of cones over the faces of $\check\Xi$, so an MPCP triangulation of $\check\shT$ gives a refinement of $\Sigma$ whose associated toric variety $\PP$ is a maximal projective partial crepant resolution $\rho:\PP\ra\PP_\Xi$. 
The inverse image under $\rho$ of a torus in $\PP_\Xi$ corresponding to a face $\check F\subset\check\Xi$ is the union of torus orbits in $\PP$ that correspond to the simplices of $\shT$ whose relative interior is contained in the relative interior of $\check F$.

\begin{proposition}
\label{Batyrev-setup} 
Let $\Xi,\check\Xi$ be a dual pair of reflexive four-dimensional polytopes and $\shT,\check\shT$ MPCP triangulations of $\partial\Xi,\partial\check\Xi$ respectively.
Let $X_1,X_2$ be the pullback under $\rho$ of general hyperplane sections in $\PP_\Xi$ and let $X_0$ be the toric boundary divisor in $\PP$ 
then $X_0,X_1,X_2$ are linearly equivalent and satisfy Assumption~\ref{mainsetup}, so give rise to $\shX,\shX',\shD,D_0$ and a perverse curve $(Z,\shF_Z)$.
By duality, an analogous constructions can be made when replacing $\Xi$ by $\check \Xi$ giving a perverse curve $(\check Z,\shF_{\check Z})$. 
We say $(Z,\shF_Z)$ and $(\check Z,\shF_{\check Z})$ are a pair of perverse curves from the Batyrev construction.
\end{proposition}

\begin{proof} The linear equivalence follows because $X_0,X_1,X_2$ are proper transforms of linearly equivalent hypersurfaces.
By \cite[Cor. 4.2.3]{Ba94}, $X_1$ and $X_2$ are smooth. The reasoning of loc.cit. is that $\PP$ has at most isolated singularities and general hypersurfaces don't meet these.
The MPCP property requires the maximal cones in the fan of $\PP$ to be simplicial cones that are cones over lattice simplices whose only integral points are its vertices. In particular the facets of such cones will be cones over 2-simplices with this property and such are isomorphic to standard cones $\RR_{\ge 0}^3$ with the usual integer lattice. Since $X_0$ is locally given by such facets, we deduce that it is normal crossing.
That various intersections between $X_0,X_1,X_2$ are normal crossing also follows from the generality assumption on $X_1,X_2$ which ensures that they meet each other and $X_0$ transversely.
\end{proof}

\subsection{Proof of Thm~\ref{maintheorem}}
We are going to prove (1),(2),(3) of Thm.~\ref{maintheorem} separately in Thm.~\ref{mainthm-1}, Thm.~\ref{mainthm-2}, Thm.~\ref{mainthm-3}.
\begin{lemma} 
\label{lem_h1=0}
Let $\PP,X_1,X_2$ be as in Prop.~\ref{Batyrev-setup}.
Let $D$ be the intersection $X_1\cap X_2$ then
$H^1(D,\CC)=0$.
\end{lemma}
\begin{proof} We set $\bar D=\rho(D)$.
By Poincar\`e duality, we may as well show $H^3(D,\CC)=0$.
Let $T\subset\PP$ be the open torus orbit and $B$ its complement.
By the long exact sequence 
$$..\ra H_c^3( D\cap T)\ra H_c^3( D) \ra H_c^3( D\cap B)\ra..$$
and the vanishing of $H_c^3( D\cap B)$ (as $ D\cap B$ is a curve)
it suffices to show that the map $H_c^3( D\cap T)\ra H_c^3( D)$ is trivial.
Indeed, the Hodge structure in the target is pure of weight $3$ but that of $H_c^3( D\cap T)$ is concentrated in weight two. 
This follows from $H_c^3( D\cap T)=H_c^3(\bar D\cap T)$, the isomorphism $H_c^3(\bar D\cap T)\ra H_c^7(T)$ of Hodge type $(2,2)$ given by Bernshte\v{\i}n's Lefschetz Theorem \cite[Thm. 6.4]{DK86} and finally 
$H_c^7(T)=H^1(T)^*$ is pure of type $(3,3)$ as $H^1(T)$ is pure of type $(1,1)$. 
The latter follows from the K\"unneth formula and the fact that $H^1(\CC^*)$ and $H^0(\CC^*)$ are pure of type $(1,1)$ and $(0,0)$ respectively which in turn can be deduced via compactifying $\CC^*$ to $\PP^1$.
Note that in general $h^{p,q}H^i(U)=h^{n-p,n-q}H_c^{2n-i}(U)$, for smooth $U$ of dimension $n$, e.g. via \cite[1.4 f)]{DK86}.
\end{proof}

Let $\shD$ denote the degenerating family of surfaces $\shX'\cap X_1$ via the construction in \S\ref{S-linsys-to-pc} applied to the setup in Prop.~\ref{Batyrev-setup}.
Let $D_0$ be the central fibre of $\shD$ and $\Gamma$ denote the $1$-skeleton of the dual intersection complex of $D_0$, i.e. $\Gamma$ is a graph with a vertex for each component of $D_0$ and an edge between two vertices if and only if the corresponding components meet in a curve.

\begin{proposition} 
\label{prop-rk-delta}
Let $b_1(\Gamma)$ denote the first Betti number of $\Gamma$ and let $\delta^*$ be the map $H^0(D^3)\ra H^2(D^2)$ in the $E_1$-term of the weight filtration \eqref{E2weightss} for the cohomology of the perverse curve constructed from $\shD$. 
We have
$$\rank\delta^* = b_1(\Gamma).$$
\end{proposition}

\begin{proof} 
The map $\delta^*$ is Poincar\'e dual to the alternating restriction map $\delta$ from \eqref{defdelta} that fits in a sequence 
\begin{equation}
\label{dual-int-sequence}
H^0(D^1)\stackrel{\delta'}{\ra} H^0(D^2)\stackrel{\delta}{\ra} H^0(D^3)
\end{equation}
where $D^1$ denotes the disjoint union of the components of $D_0$ and $\delta'$ an alternating restriction map defined similar to \eqref{defdelta}, see \cite[Lemma 4.7,(4)]{GKR12}. 
This sequence computes the cohomology of the dual intersection complex of $D_0$. It also appears as the bottom row $(m+k=0)$ of the $E_1$-term 
\begin{equation}
E_1^{-k,m+k}= \bigoplus_{q>0,-k} H^{m-2q-k}(D^{2q+k+1},\CC)\langle -q-k\rangle\Rightarrow H^m(D_0,\psi_f\CC)
\label{wss-nearby}
\end{equation}
of the weight spectral sequence computing the cohomology of the nearby fibre of $D_0$.
See \cite[Fig. 8]{GKR12} which shows a version of this for degenerations of threefolds, see \cite[(XI--29)]{PS08} for a reference of \eqref{wss-nearby}. 
It degenerates at $E_2$ by \cite[III, Scholie (8.1.9), (iv)]{DelTH}.
The cohomology $\ker\delta / \im\delta'$ is a direct summand of the first cohomology of the nearby fibre and thus vanishes by Lemma~\ref{lem_h1=0}. 
We conclude that 
$$e-\rank\delta^* = \dim\coker\delta^* = \dim\ker\delta = \rank\delta'$$
where $e=\dim H^2(D^2)=\dim H^0(D^2)$.
Note that $e$ is the number of edges of $\Gamma$ and $v=\dim H^0(D^1)$ the number of vertices.
As the dual intersection complex is connected, the kernel of $\delta'$ has rank one and thus $\rank\delta'=v-1$. 
We conclude that 
\begin{equation}
\rank\delta^* = 1+e-v
\label{1+e-v}
\end{equation}
which coincides with $b_1(\Gamma)$ as $\Gamma$ is connected.
\end{proof}

\begin{theorem} 
\label{mainthm-2}
Let $(Z,\shF_Z)$ be a perverse curve constructed via Prop.~\ref{Batyrev-setup}.
Let $v,e$ be the number of vertices and edges of the dual intersection complex of $D_0$. Note that $e$ coincides with the number of components of $Z$.
Let $n$ be the number of triple points in $Z$ and $g$ the sum of the genera of the components of $Z$. We have
$$h^{1,0}(Z,\shF_Z)= h^{0,1}(Z,\shF_Z)=n+g-b_1(\Gamma)=n+g+v-e-1,$$
$$h^{0,0}(Z,\shF_Z)= h^{1,1}(Z,\shF_Z)=e-b_1(\Gamma)=v-1.$$
\end{theorem}

\begin{proof} 
This follows from Lemma~\ref{Poincare-pc}, Prop.~\ref{Hodge-pc}, Prop.~\ref{prop-rk-delta} and \eqref{1+e-v}.
\end{proof}

Recall the following well-known fact.\\
\begin{minipage}[b]{0.70\textwidth} 
\begin{lemma} 
A curve of genus $g\ge 2$ decomposes into $2g-2$ many pairs of pants by removing $3g-3$ suitably chosen disjoint circles from it.
\end{lemma}
\end{minipage}\qquad
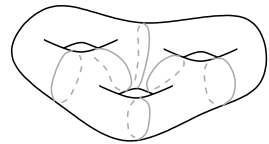
\begin{minipage}[t]{0.26\textwidth}
\resizebox{0.9\textwidth}{!}{
\begin{picture}(0,0)%
\includegraphics{decompose.pstex}%
\end{picture}%
\setlength{\unitlength}{4144sp}%
\begingroup\makeatletter\ifx\SetFigFont\undefined%
\gdef\SetFigFont#1#2#3#4#5{%
  \reset@font\fontsize{#1}{#2pt}%
  \fontfamily{#3}\fontseries{#4}\fontshape{#5}%
  \selectfont}%
\fi\endgroup%
\begin{picture}(1269,685)(5366,-5158)
\end{picture}%

}
\end{minipage}\\[4mm]
Thus, degenerating a smooth curve of genus $g\ge2$ into a stable curve with components being $\PP^1$s we find that 
$$g=d-c+1$$ 
where $d=3g-3$ is the number of double points and $c=2g-2$ the number of components. 
This formula also holds for genus $g=0,1$ if we degenerate an elliptic curve into a chain of $\PP^1$s. 
Also a further degeneration by contracting a circle in a $\PP^1$ preserves this formula.
Degenerating each component of $Z$ yields
\begin{equation}
\label{eq-tildee}
h^{1,0}(Z,\shF_Z)= h^{0,1}(Z,\shF_Z)=n+d+v-\tilde e,
\end{equation}
where $d$ is the total count of all double points and $\tilde e$ is the total number of $\PP^1$s.

\begin{lemma}
\label{countZdata} 
Let $\shT$ be a MPCP triangulation of $\partial\Xi$ and $\check \shT$ be a MPCP triangulation of $\partial\check \Xi$. 
For a face $F\subset\partial\Xi$ (resp. $F\subset\partial\check\Xi$) and $i\ge0$, let $s^i(F)$ denote the number of $i$-dimensional simplices in $\shT$ (resp. $\check\shT$) which intersect the relative interior of $F$ non-trivially. For an edge $F\subset\Xi$ we also use the notation $\len(F)=s^1(F)$.
\begin{enumerate}
\item For $F\subset\partial\Xi$ (resp. $\subset\partial\check \Xi$) a $2$-face, $s^i(F)$ is independent of $\shT$ (resp. $\check\shT$) for $i=0,1,2$. Moreover,
$$ s^0(F) - s^1(F) + s^2(F) = 1.$$
\item $$\tilde e=\sum_{{F\subset\Xi}\atop{\dim F=2}} s^2(F) \len(\check F) + \sum_{{\check F\subset\check\Xi}\atop{\dim \check F=2}}  s^1(\check F) \len(F)$$
\item $$d+n=\sum_{{F\subset\Xi}\atop{\dim F=2}} s^1(F) \len(\check F) + \sum_{{\check F\subset\check\Xi}\atop{\dim \check F=2}}  s^2(\check F) \len(F)$$
\item $$n+d-\tilde e=\sum_{{F\subset\Xi}\atop{\dim F=2}} (s^0(F)-1) \len(\check F) - \sum_{{\check F\subset\check\Xi}\atop{\dim \check F=2}}  (s^0(\check F)-1) \len(F) $$
\item $$v=\#\{\hbox{vertices of }\check\Xi\}+ \sum_{{F\subset\Xi}\atop{\dim F=2}} s^0(\check F) + \sum_{{\check F\subset\check \Xi}\atop{\dim\check F=2}} s^0(\check F)\len(F)  $$
\end{enumerate}
\end{lemma}
\begin{proof} 
Note that $s^0(F) - s^1(F) + s^2(F)$ computes the Euler number of relative homology $H_\bullet(B^2,\partial B^2;\ZZ)$ of a $2$-ball $B^2$ relative to its boundary and this is one.
The first statement in (1) then follows because $s^0(F)$ is the number of interior lattice points and $s^2(F)$ is determined by the lattice volume of $F$, so we are done with (1).

Note that the subdivision $\check \shT$ of $\check\partial\Xi$ determines a subdivision of the fan of the toric variety $\PP_\Xi$ associated to $\Xi$ giving the MPCP resolution $\pi:\PP\ra\PP_\Xi$, see \cite[Thm. 2.2.24]{Ba94}.
By standard toric geometry, the $i$-dimensional torus orbits of $\PP_\Xi$ are indexed by $i$-faces $F\subset\Xi$, let $O_F$ be the orbit indexed by $F$. 
Similarly via the fan combinatorics, $i$-dimensional orbits of $\PP$ are indexed by $(3-i)$-dimensional cells $\tau$ of $\check\T$ and we denote the orbit corresponding to $\tau$ by $O_\tau$. We then have
\begin{equation} \label{inverseorbits}
\pi^{-1}(O_F)=\bigcup_{\tau\hbox{ \scriptsize meets the interior of }\check F}O_\tau.
\end{equation}
Note that for $F\subset\Xi$ and $i$-cell, $\check F$ is a $(3-i)$-cell.
Let $\bar Z=\pi(Z)$ denote the blowdown of $Z\subset\PP$ under $\pi$.
Note that $\bar Z$ is contained in the union of one- and two-dimensional torus orbits of $\PP_\Xi$ as it is the intersection of the singular locus of the boundary divisor with a general section of $\shO_{\PP_\Xi}(1)$.
These orbits thus correspond to edges and $2$-faces of $\Xi$.
In general, $\pi^{-1}({\bar Z})\neq Z$ since $Z$ is require to lie in the singular locus of the boundary divisor. Given $F\subset\Xi$ of dimension $1$ or $2$, we conclude from \eqref{inverseorbits-Z}
\begin{equation}
\label{inverseorbits-Z}
\pi^{-1}({\bar Z}\cap O_F)\cap Z=\bigcup_{{\tau\hbox{ \tiny meets the interior of }\check F}\atop{\dim O_\tau<3}}O_\tau.
\end{equation}
Note that $\dim O_\tau<3 \iff \dim\tau>0$.

When $\dim F=2$ then $\dim\check F=1$ and $\pi^{-1}({\bar Z}\cap O_F)$ is a disjoint union of $\len(\check F)$ copies of ${\bar Z}\cap O_F$.
Since $F$ is the Newton polytope of ${\bar Z}\cap O_F$ and this is triangulated in $s^2(F)$ triangles, we find ${\bar Z}\cap O_F$ decomposes in $s^2(F)$ many pairs of pants. We have deduced the first sum in (2).

When $\dim F=1$ then ${\bar Z}\cap O_F$ is a union of $\len(F)$ many point because $F$ is the Newton polytope of this set of points, let $p$ be one of these points.
We have $\dim\check F=2$ and $\pi^{-1}(p)$ is a union of $s^1(\check F)$ many $\PP^1$ that contains $s^2(\check F)$ zero-dimensional orbits (e.g. intersection points of them). We thus deduce the second sum in (2) and are done with (2).

Analogous to (2), the first sum in (3) counts the double points coming from the circles in a pair of pants decomposition of the components of $\bar Z$, as before these are multiplied by $\len(\check F)$ under taking $\pi^{-1}$. The second sum in (3) counts triple points of $Z$ that map to points in $\bar Z$.

(4) is obtained by applying the formula in (1) to the difference of (3) and (2).

Finally, the first summand in (5) counts components of $\pi(D_0)$ (these correspond to components of $D_0$ under pullback), the second sum gives the components of $D_0$ mapping to curves under $\pi$ and the last sum those that map to a point of $\pi(D_0)$.
\end{proof}

\begin{theorem} \label{mainthm-1}
The Euler number of $(Z,\shF_Z)$ coincides with that of a general hypersurface in $\PP$.
\end{theorem}  
\begin{proof} 
By \cite[Thm. 4.5.3]{Ba94}, the Euler number of an anti-canonical hypersurface in $\PP$ is 
$$\sum_{{F\subset\Xi}\atop{\dim F=2}} \vol(F) \len(\check F) - \sum_{{\check F\subset\check\Xi}\atop{\dim \check F=2}}  \vol(\check F) \len(F) $$
multiplied by $2$ where $\vol(F)$ denotes the lattice volume of $F$. Pick's theorem states that
$$\vol(F) = s^0(F)+\frac{b(F)}2-1$$
where $b(F)$ denotes the number of lattice points in the boundary of $F$.
On the other hand, by Thm.~\ref{mainthm-2} and \eqref{eq-tildee}, the Euler number of $(Z,\shF_Z)$ is $(-2)$ times the expression in (4) of Lemma~\ref{countZdata}. It remains to show that
$$ 0 = \sum_{{F\subset\Xi}\atop{\dim F=2}} b(F) \len(\check F) - \sum_{{\check F\subset\check\Xi}\atop{\dim \check F=2}}  b(\check F) \len(F). $$
Indeed $b(F)$ coincides with the number of edges of $\shT$ contained in $\partial F$, so the last equation is equivalent to
$$ 0 = \sum_{{e\subset F\subset\Xi}\atop{\dim F=2,\dim e=1}} \len(e) \len(\check F) - \sum_{{\check e\subset \check F\subset\check\Xi}\atop{\dim \check F=2,\dim\check e=1}}  \len(\check e) \len(F) $$
where the two sums agree (up to sign) by duality.
\end{proof}

\begin{theorem} \label{mainthm-3}
$h^{p,q}(Z,\shF_Z)= h^{1-p,q}(\check Z,\shF_{\check Z}).$
\end{theorem}

\begin{proof} By symmetry and Lemma~\ref{Poincare-pc}, it suffices to show $h^{0,1}(Z, \shF_Z) = h^{0,0}(\check Z, \shF_{\check Z})$. 
For this, we use \eqref{eq-tildee} where we insert (4) and (5) of Lemma~\ref{countZdata} and use that for an edge $\check F$ holds $s^0(\check F)=\len(\check F)-1$ to obtain
$$h^{0,1}(Z, \shF_Z)= \#\{\hbox{vertices of }\check\Xi\}+\sum_{{F\subset\Xi}\atop{\dim F=2}} (s^0(F)\len(\check F)-1) + \sum_{{\check F\subset\check\Xi}\atop{\dim \check F=2}}  \len(F).$$
We want to identify this with $h^{0,0}(\check Z, \shF_{\check Z})$ that takes the form 
$$h^{0,0}(\check Z, \shF_{\check Z})=\#\{\hbox{vertices of }\Xi\}+ \sum_{{\check F\subset\check\Xi}\atop{\dim \check F=2}} s^0( F) + \sum_{{F\subset \Xi}\atop{\dim F=2}} s^0( F)\len(\check F)$$
via Thm.~\ref{mainthm-2} and (5) of Lemma~\ref{countZdata}. For this one uses that for an edge $F$ holds $s^0(F)=\len(F)-1$ together with the identity
$$ \#\{\hbox{vertices of }\check\Xi\} + \sum_{{F\subset\Xi}\atop{\dim F=2}} (-1) =  \#\{\hbox{vertices of }\Xi\} + \sum_{{\check F\subset\check\Xi}\atop{\dim \check F=2}} (-1) $$
that can be derived from the computation of the vanishing Euler number of $\partial\Xi \cong S^3$ using the duality of faces of $\Xi$ and $\check\Xi$.
\end{proof}

\section{Appendix}
\label{appendix}
We add here the calculation of the Hodge diamonds in Examples~\ref{examples},\,(1)-(2) using Theorem~\ref{maintheorem}. By part (3) it suffices to compute the Hodge numbers of one mirror partner.
\begin{figure}
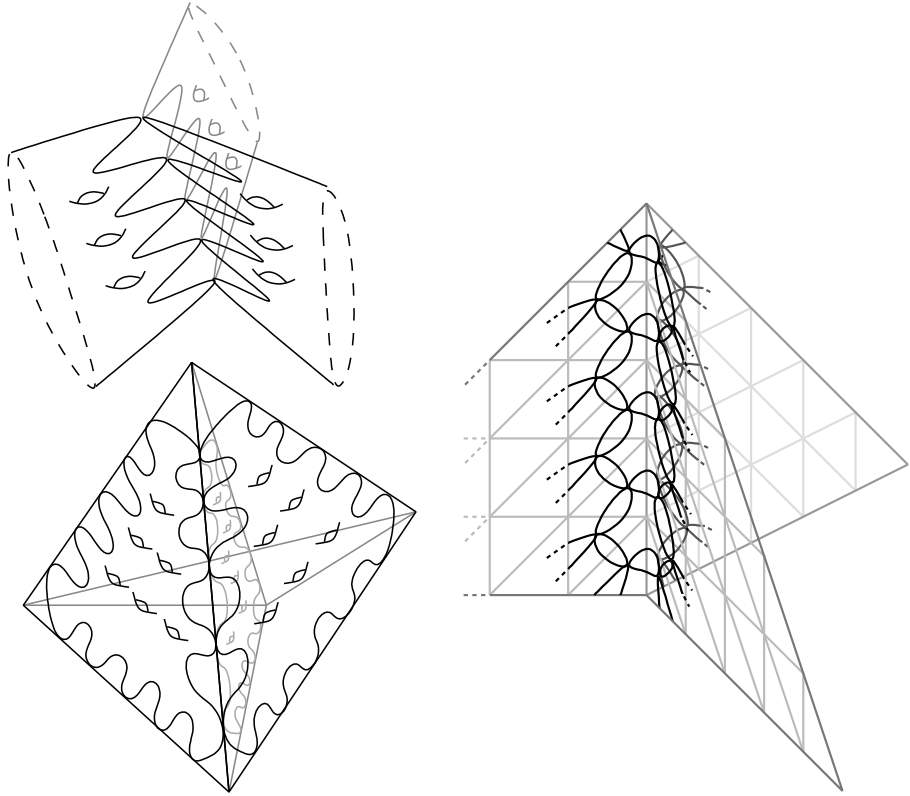

\resizebox{0.36\textwidth}{!}{
\input{quintic_curves.pstex_t}
}
\resizebox{0.4\textwidth}{!}{
\input{quintic-mirror.pstex_t}
}
\caption{Part of the perverse curve in the quintic threefold degeneration on the left and part of its mirror dual on the right.}
\label{quintic-pcs}
\end{figure}
Consider a general quintic threefold $X\subset\PP^4$ and let $D_0$ be its intersection with the toric boundary and $Z=\Sing D_0$. Note that the number of components of $D_0$ is $5 =:v$.
Then $Z$ consists of the union of $\binom53=10=:e$ smooth quintic plane curves (i.e. genus six) each given by the intersection of a coordinate $\PP^2$ with the quintic hypersurface. The sum of the genera is $g=60$.  
Since each of the $\binom52=10$ coordinate $\PP^1$s intersects the quintic in $5$ points and each coordinate $\PP^1$ is contained in $3$ coordinate $\PP^2$s, we conclude that the numbers of triple points is $n=50$. 
By Thm.~\ref{maintheorem},\,(2) we have
$$
\begin{array}{rcccccccl}
h^{1,0}(Z,\shF_Z) &=& h^{0,1}(Z,\shF_Z) &=& v-1+n+g-e &=& 5-1+50+60-10 &=& 104, \\
h^{0,0}(Z,\shF_Z) &=& h^{1,1}(Z,\shF_Z) &=& v-1 &=& 5-1 &=&4.$$ 
\end{array}
$$
We next consider a general anticanonical hypersurface $X$ in $(\PP^1)^4$ and let again $D_0$ be its intersection with the toric boundary. This has $v=2^1\binom41=8$ components each being an intersection of a toric prime divisor with $X$. The components of $Z=\Sing D_0$ are the intersection of the coordinate $(\PP^1)^2$s of which there are $e=2^2\binom42=24$. Each component of $Z$ is a genus one curve and the sum of their genera is thus $g=24$. The set of triple points $n$ in $Z$ is the number of intersection points of $X$ with the coordinate $\PP^1$s of which we have $e=2^3\binom43=32$ and each contributes two points, so $n=64$.
We conclude via Thm.~\ref{maintheorem},\,(2)
$$
\begin{array}{rcccccccl}
h^{1,0}(Z,\shF_Z) &=& h^{0,1}(Z,\shF_Z) &=& v-1+n+g-e &=& 8-1+64+24-24 &=& 71, \\
h^{0,0}(Z,\shF_Z) &=& h^{1,1}(Z,\shF_Z) &=& v-1 &=& 8-1 &=&7.$$ 
\end{array}
$$

\end{document}